\documentclass[11pt]{ip-journal}
\usepackage{amscd,amssymb}
\usepackage[mathscr]{eucal}
\usepackage[english]{babel}
\input pstricks
\input pst-node

\usepackage{pstricks, pst-node}
\usepackage[all]{xy}

\usepackage{amsmath}

\usepackage[latin1]{inputenc}%

\newlength{\itemlaenge}

\newtheoremstyle{mytheorem}
  {}
  {}
  {\slshape}
  {}
  {\scshape}
  {.}
  { }
  {}

\newtheoremstyle{mydefinition}
  {}
  {}
  {\upshape}
  {}
  {\scshape}
  {.}
  { }
  {}

\theoremstyle{mytheorem}
\newtheorem{lemma}{Lemma}[section]
\newtheorem{prop}[lemma]{Proposition}
\newtheorem*{prop*}{Proposition}

\newtheorem{cor}[lemma]{Corollary}

\newtheorem{thm}[lemma]{Theorem}

\newtheorem*{thm*}{Theorem}
\theoremstyle{mydefinition}
\newtheorem{rem}[lemma]{Remark}
\newtheorem*{rem*}{Remark}

\newtheorem*{notation*}{Notation}

\newtheorem*{warning*}{Warning}

\newtheorem{defi}[lemma]{Definition}
\newtheorem*{defi*}{Definition}

\numberwithin{equation}{section}

\newcommand{\bqn}{\begin{equation*}}
\newcommand{\eqn}{\end{equation*}}
\newcommand{\bq}{\begin{equation}}
\newcommand{\eq}{\end{equation}}
\newcommand{\ba}{\begin{aligned}}
\newcommand{\ea}{\end{aligned}}
\newcommand{\be}{\begin{enumerate}}
\newcommand{\ee}{\end{enumerate}}

\newcommand{\thismonth}{\ifcase\month 
  \or January\or February\or March\or April\or May\or June%
  \or July\or August\or September\or October\or November%
  \or December\fi}

\newcommand{\aut}{\operatorname{Aut}}

\newcommand{\SL}{\operatorname{SL}}
\newcommand{\Sp}{\operatorname{Sp}}

\newcommand{\PU}{\operatorname{PU}}

\newcommand{\CC}{{\mathbb C}}
\newcommand{\DD}{{\mathbb D}}

\newcommand{\NN}{{\mathbb N}}

\newcommand{\QQ}{{\mathbb Q}}
\newcommand{\RR}{{\mathbb R}}

\newcommand{\ZZ}{{\mathbb Z}}

\newcommand{\Hh}{{\mathcal H}}

\newcommand{\Rr}{{\mathcal R}}

\renewcommand{\b}{\beta}

\renewcommand{\L}{\Lambda}
\renewcommand{\l}{\lambda}

\newcommand{\gG}{{\mathbf G}}

\newcommand{\lL}{{\mathbf L}}

\newcommand{\<}{\langle}
\renewcommand{\>}{\rangle}

\newcommand{\rad}{\mathrm{Rad}}

\def\h{{\rm H}}
\def\hb{{\rm H}_{\rm b}}

\def\hcb{{\rm H}_{\rm cb}}

\def\linfty{L^\infty}

\def\binfty{\mathcal B^\infty_{\mathrm alt}}
\def\T{\operatorname{T}}

\def\one{\mathbf{1\kern-1.6mm 1}}

\def\homeoz#1{\operatorname{Homeo}^+_\ZZ\!\left(#1\right)}

\def\c{{\operatorname{c}}}

\def\h2{{\operatorname{H_2}}}
\def\h1{{\operatorname{H_1}}}

\def\res{{\operatorname{Res}}}
\def\rk{{\operatorname{rank}}}

\def\ker{{\operatorname{ker}}}
\def\im{{\operatorname{im}}}

\def\SL{\operatorname{SL}}
\def\Sp{\operatorname{Sp}}

\def\cs{{\check S}}

\def\eb{{e^{\rm b}}}

\def\ksib{\kappa_\Sigma^{\rm b}}

\def\kpub{{\kappa^{\rm b}_{\PU(1,1)}}}

\def\khb{\kappa_{H}^{\rm b}}

\def\k{\kappa}

\def\kgb{\kappa_G^{\rm b}}

\def\khb{\kappa_{H}^{\rm b}}

\def\tb{{\rm t}_{\rm b}}

\def\binfty{\mathcal B^\infty_{\mathrm {alt}}}
\def\hcb{{\rm H}_{\rm cb}}
\def\to{\rightarrow}

\def\hb{{\rm H}_{\rm b}}

\def\h{{\rm H}}

\def\hhc{\hat{\rm H}_{\rm c}}
\def\hhcb{\hat{\rm H}_{\rm cb}}

\def\vk{\varkappa}
\def\vkj{\vk_{H_j}}
\def\vkr{{\varkappa^\RR}}
\def\vkrl{{\varkappa^\RR_L}}
\def\vkrs{{\varkappa^\RR_S}}
\def\vkrj{\vk^{\RR}_{H_j}}

\renewcommand{\phi}{\varphi}

\def\No{N\raise4pt\hbox{\tiny o}\kern+.2em}
\def\no{n\raise4pt\hbox{\tiny o}\kern+.2em}

\renewcommand{\O}{\textup{O}}

\newcommand{\orn}{\textup{or}}

\renewcommand{\hom}{\textup{Hom}}
\newcommand{\homc}{\textup{Hom}_{\textup c}}

\newcommand{\Rad}{\textup{Rad}}

\begin{document}

\title[Weakly maximal representations]{On weakly maximal representations\\ of surface groups}
\author[Ben Simon]{G.~Ben Simon}
\email{gabi.ben-simon@math.ethz.ch}
\address{Department Mathematik, ETH Z\"urich, R\"amistrasse 101, CH-8092 Z\"urich, Switzerland}
\author[Burger]{M.~Burger}
\email{burger@math.ethz.ch}
\address{Department Mathematik, ETH Z\"urich, R\"amistrasse 101, CH-8092 Z\"urich, Switzerland}
\author[Hartnick]{T.~Hartnick}
\email{hartnick@tx.technion.ac.il}
\address{Mathematics Department, Technion - Israel Institute of Technology, Haifa, 32000, Israel}
\author[Iozzi]{A.~Iozzi}
\email{iozzi@math.ethz.ch}
\address{Department Mathematik, ETH Z\"urich, R\"amistrasse 101,
  CH-8092 Z\"urich, Switzerland}
\author[Wienhard]{A.~Wienhard}
\email{wienhard@mathi.uni-heidelberg.de}
\address{Mathematisches Institut, Ruprecht-Karls-Universit\"at Heidelberg, 69120 Heidelberg, Germany
\newline
\newline HITS gGmbH, Heidelberg Institute for Theoretical Studies, Schloss-Wolfs\-brunnen\-weg 35, 69118 Heidelberg, Germany}
\thanks{M.~B. was partial supported by the Swiss National Science Foundation project  200020-144373; 
T.~H. was partial supported by the Swiss National Science Foundation project 2000021-127016/2; 
A.~I. was partial supported by the Swiss National Science Foundation projects 2000021-127016/2 and 200020-144373;
A.~W. was partially supported by the National Science Foundation under agreement No. DMS-1065919 and 0846408, by the Sloan Foundation, by the Deutsche Forschungsgemeinschaft, and by the ERCEA under ERC-Consolidator grant no. 614733.  
Support by the Institut Mittag-Leffler (Djursholm, Sweden) and by the Institute for Advanced Study (Princeton, NJ) is gratefully
acknowledged.}


\date{\today}

\begin{abstract} 
We introduce and study a new class of representations of surface groups into Lie groups of Hermitian type, 
called {\em weakly maximal} representations.  
We prove that weakly maximal representations are discrete and injective and 
we describe the structure of the Zariski closure of their image.  
Furthermore we prove that the set of weakly maximal representations is a closed subset of the representation variety and describe its relation to 
other geometrically significant subsets of the representations variety.
\end{abstract}
\maketitle
%
%
%
%


\section{Introduction}\label{sec:intro}

Given a compact oriented surface $\Sigma$ of negative Euler characteristic, 
possibly with boundary, a general theme is to study the space of representations
$\hom(\pi_1 (\Sigma), G)$ of the fundamental group $\pi_1(\Sigma)$ of $\Sigma$ 
into a semisimple Lie group $G$, and in
particular to distinguish subsets of geometric significance.

In recent years these studies led to the discovery of new subsets of 
representation varieties:  Hitchin
components \cite{Hitchin, Goldman_Choi, Labourie_anosov, Guichard_convex, Guichard_Wienhard_convex}, 
positive representations \cite{Fock_Goncharov,Fock_Goncharov_convex}, maximal representations
\cite{Goldman_thesis, Goldman_82, Toledo_89, Hernandez,
Burger_Iozzi_Wienhard_tol, Burger_Iozzi_Wienhard_anosov,
Burger_Iozzi_Labourie_Wienhard, Burger_Iozzi_Wienhard_htt, Wienhard_mapping,
Hartnick_Strubel, Gothen, Bradlow_GarciaPrada_Gothen_survey, Bradlow_GarciaPrada_Gothen,
Bradlow_GarciaPrada_Gothen_sp4, GarciaPrada_Gothen_Mundet} and Anosov representations
\cite{Guichard_Wienhard_anosov, Labourie_anosov, Guichard_Wienhard_invariants}. 
Even though these subsets are defined and investigated by very different methods, they 
exhibit several common properties, and  their study is summarized under the terminology \emph{higher Teichm\"uller theory}.\\

Here we introduce a new class of representations of $\pi_1(\Sigma)$ into a Lie group of Hermitian type, 
the set of {\em weakly maximal representations}. 
We establish several results showing that weakly maximal representations are of geometric significance. 
We also construct various nontrivial example of weakly maximal representations, 
establish a relation to orders on Lie groups, 
and discuss the relation of the set of weakly maximal representations with other subsets of the representation variety.  

Weakly maximal representations with non-zero Toledo number also admit a geometric characterization. 
This geometric characterization is discussed in detail in the companion paper \cite{BBHIW2}. 
We show there that weakly maximal representations with non-zero Toledo number are representations 
that are order preserving in an appropriate sense with respect to an arbitrary bi-invariant continuous partial order on the cyclic covering of $G$. 
If the symmetric space associated to $G$ is a Hermitian symmetric space of tube type, 
they are also  order-preserving with respect to the order induced by the Kaneyuki causal ordering on the corresponding Shilov boundary. 
This leads to a particular simple geometric characterization in the tube type case.

\subsection{Structure Theorems for Weakly Maximal Representations}\label{1.1}
Recall that the Toledo invariant $\T(\rho)$ of a representation $\rho:\pi_1(\Sigma)\to G$ 
is defined as the evaluation, in an appropriate sense, of the pullback $\rho^*(\kgb)$
of the bounded K\"ahler class $\kgb\in\hcb^2(G,\RR)$ on the relative fundamental class
$[\Sigma,\partial\Sigma]$ (see \S~\ref{sec:tolinv} for details).  This definition in terms of bounded
cohomology leads to a chain of inequalities
\begin{eqnarray*}
 |\T(\rho)|\leq 2\|\rho^*(\kgb)\| \cdot |\chi(\Sigma)| \leq 2\|\kgb\| \cdot
|\chi(\Sigma)|\,,
\end{eqnarray*}
where $\|\,\cdot\,\|$ denotes the canonical norm on the Banach spaces $\hcb^2(G,\RR)$
and $\hb^2(\pi_1(\Sigma),\RR)$, 
while $|\chi(\Sigma)|$ appears as (half of) the $\ell^1$-simplicial norm of $[\Sigma,\partial\Sigma]$.
{\em Maximal representations} are those for which both inequalities are equalities,
while representations for which the second inequality is an equality  are {\em tight homomorphisms};
these constitute a very interesting class of homomorphisms which can be defined and studied 
in much greater generality, see \cite{Burger_Iozzi_Wienhard_tight, Hamlet_tighthol, Hamlet_tight}.  
In this article we set ourselves the goal to analyse 
the structure of representations for which the first inequality is an equality.

%

\begin{defi}[{\cite[Chapter~8, Definition~2.1]{Wienhard_thesis}}]
A representation $\rho: \pi_1(\Sigma) \to G$ is \emph{weakly maximal} if it
satisfies the equality
\begin{eqnarray}\label{WMDefIneq}
\<\rho^*(\kgb),[\Sigma,\partial\Sigma]\>=\|\rho^*(\kgb)\|_\infty\|\Sigma,\partial\Sigma\|_1
\end{eqnarray}
in the standard inequality of dual norms.
\end{defi}

\begin{rem}
Notice that the left hand side of \eqref{WMDefIneq} is the Toledo invariant $\T(\rho)$,
while the right hand side equals $2\|\rho^*(\kgb)\| \cdot |\chi(\Sigma)|$.

By definition a weakly maximal representation has non-negative Toledo number. 
Analogously one can define weakly minimal representations. Since the composition 
of  a representation with an orientation reversing outer automorphism changes 
the sign of the Toledo number the theory is completely analogous. 
\end{rem}


Whereas for maximal representations the Toledo invariant is a fixed number independent of $\rho$, for weakly maximal representations the right hand side of the equality is $ 2\|\rho^*(\kgb)\| \cdot |\chi(\Sigma)|$ and depends on $\rho$. In particular, there is (a priori) no restriction on the Toledo invariant of a weakly maximal representation.  

In order to get a better understanding of weakly maximal representations it is therefore important to give alternative characterizations. 
One such characterization is that a representation $\rho$ is weakly maximal if and only if there exists
$\lambda \geq 0$ such that 
\bq\label{eq:lambda}
\rho^*(\kgb) = \lambda \, \kappa_\Sigma^{\mathrm b}\,,
\eq
where $\ksib$ if the bounded fundamental class of $\Sigma$ (see \eqref{eq:bfc}).
This characterization was in effect first established in \cite[Corollary~3.4]{Wienhard_thesis} for compact 
surfaces and in \cite[Cor. 4.15]{Burger_Iozzi_Wienhard_tol} in the general case; 
a different proof using Bavard's duality was later obtained in \cite{Calegari}
(for a more thorough description of the relation with Calegari's work, 
see \cite[Section~4.6]{Burger_Iozzi_Wienhard_htt}).

The constant $\lambda$ appearing in \eqref{eq:lambda} is related to the Toledo invariant
by 
\bqn
\T(\rho)=\lambda\,|\chi(\Sigma)|\,.
\eqn 

When $\Sigma$ is a surface without boundary the Toledo invariant is a characteristic number with values in $ q_G^{-1}\ZZ$, where $q_G$ is a natural number depending only on $G$
(namely, $q_G$ is the smallest integer such that $q_G\kgb$ is an integral class, see Remark~\ref{rem:constants}) . On the other hand, when $\Sigma$ is a surface with boundary, the Toledo invariant ranges over the whole closed interval
$\big[-2\|\kgb\|\cdot |\chi(\Sigma)| , 2\|\kgb\|\cdot |\chi(\Sigma)|\big]$.

It is remarkable that for weakly maximal representations we can restrict the possible values of the Toledo-invariant by the following 
\begin{thm}[Rationality Theorem]\label{thm_intro:integer}
There is a natural number $\ell_G$ depending only on $G$, such that for every weakly
maximal representation $\rho: \pi_1(\Sigma) \to G$ we have $\T(\rho) \in \frac{|\chi(\Sigma)|}{\ell_G}\ZZ$. 
In particular, $\T$ takes only finitely many values on the set of weakly maximal
representations.
\end{thm}
\begin{rem} The integer $\ell_G$ depends in an explicit way on $q_G$ 
and on the degree of non-integrality of the restriction $\kgb|_H$ to various connected semisimple
subgroups $H$ of $G$, as well as on the cardinality of their center
(see the proof of Corollary~\ref{cor:4.7}).
\end{rem}

The Rationality Theorem is an essential ingredient in order to prove 
\begin{thm}\label{thm_intro: discrete_faithful}
Let $\rho: \pi_1(\Sigma) \to G$ be a weakly maximal representation and $\T(\rho)
\neq 0$. Then $\rho$ is faithful with discrete image.
\end{thm}

The methods of proof for this theorem are very different from those used in the case of maximal representations \cite{Burger_Iozzi_Wienhard_tol} and give in particular a new and simpler proof of the injectivity and discreteness of maximal representations. 

Whereas the Zariski closure for maximal representations is always reductive and subject to strong restrictions, 
the Zariski closure of the image of a weakly maximal representation can be quite wild. In particular, it  
need not be reductive (see \S~\ref{subsec:nonred}).  
However, we can establish a structure theorem describing the Zariski closure of weakly maximal representations as follows. 
Given a closed subgroup $L < G$, we show that there exists a unique maximal
normal subgroup of $L$, that we call the {\em K\"ahler radical} $\rad_{\kgb}(L)$ of $L$,
on which $\kgb|_L$ vanishes. Since $\rad_{\kgb}(L)$ contains the solvable radical of $L$,
the quotient $L/\rad_{\kgb}(L)$ is semisimple.
We then establish the following as part of Theorem \ref{thm:structure}:
\begin{thm}\label{ZariskiClosure}
Let $\rho: \pi_1(\Sigma) \to G$ be a weakly maximal representation 
into a Lie group $G:=\gG(\RR)^\circ$ of Hermitian type, 
where $\gG$ is a connected algebraic group defined over $\RR$.
Let $L:=\lL(\RR)$, where $\lL$ is the Zariski closure of $\rho(\pi_1(\Sigma))$ in $\gG$,
and let $H=L/\ \rad_{\kgb}(L)$ be the quotient of $L$ by its K\"ahler radical.  
Assume that $\T(\rho)\neq 0$. 
Then:
\begin{enumerate}
\item  $H$ is adjoint of Hermitian type and all of its simple factors
are of tube type. 
\item The composition $\pi_1(\Sigma) \to L \to H$ is faithful with discrete
image. 
\end{enumerate}
\end{thm}
Recall that a Hermitian Lie group is of {\em tube type} if the associated
symmetric space is of tube type, that is biholomorphic to $\RR^n+iC$, for some $n\in\NN$,
where $C\subset\RR^n$ is an open convex cone.

\begin{rem}\label{rem:tr}
In the above Theorems~\ref{thm_intro: discrete_faithful} and \ref{ZariskiClosure} 
it is essential that the Toledo invariant is non-zero. However
the class of weakly maximal representations with $\T(\rho) = 0$ is also of
interest. This is precisely the set where
$\rho^*(\kgb) = 0$. In the case when $G = \PU(1,n)$ such representations have
been studied and characterized in \cite{Burger_Iozzi_tr} 
as representations that preserve a totally real subspace of complex hyperbolic space $\Hh_\CC^n$.
\end{rem}

\subsection{Comparison with Other Classes of Representations}\label{1.3}
Consider the representation variety $\hom(\pi_1(\Sigma),G)$. 
%

The subset of weakly maximal representations
$\hom_{wm}(\pi_1(\Sigma), G)$ decomposes as a disjoint union
\bqn
\hom_{wm}(\pi_1(\Sigma), G)=\hom^*_{wm}(\pi_1(\Sigma), G)\sqcup\hom_0(\pi_1(\Sigma), G)
\eqn
of the set $\hom^*_{wm}(\pi_1(\Sigma), G)$ of weakly maximal representations with non-zero 
Toledo invariant and the set of representations with vanishing Toledo invariant
\bqn
\hom_0(\pi_1(\Sigma),G):=\{\rho:\pi_1(\Sigma)\to G:\,\rho^\ast(\kgb)=0\}\,.
\eqn
The latter contains $\hom(\pi_1(\Sigma),L)$ for every closed subgroup $L<G$
for which $\kgb|_L=0$.  
In particular, representations in $\hom_0(\pi_1(\Sigma),G)$ are not necessarily injective,
and their images are not necessarily discrete (see Remark~\ref{rem:tr}),
while, according to Theorem~\ref{thm_intro: discrete_faithful},
$\hom^*_{wm}(\pi_1(\Sigma), G)$ is contained in the set
$\hom_{d,i}(\pi_1(\Sigma), G)$ of injective homomorphisms
with discrete image.  
We prove (cf. Corollary \ref{cor:wm*}):

\begin{thm}\label{thm:max-wm-di-wm}  The following 
\bqn
\xymatrix{
& &\hom_{d,i}(\pi_1(\Sigma), G)\\
\hom_{max}(\pi_1(\Sigma),G)\ar@{^{(}->}[r]
&\hom^*_{wm}(\pi_1(\Sigma), G)\ar@{^{(}->}[ur]\ar@{^{(}->}[dr]\\
& & \hom_{wm}(\pi_1(\Sigma), G)
}
\eqn
is a diagram of $\aut(\pi_1(\Sigma))$-invariant closed subsets of the representation variety 
$\hom(\pi_1(\Sigma), G)$.
\end{thm}
If $G$ is real algebraic, 
$\hom_{max}(\pi_1(\Sigma), G)$ is a real semi-algebraic subset of $\hom(\pi_1(\Sigma), G)$,
\cite[Corollary~14]{Burger_Iozzi_Wienhard_tol},
but we do not have such precise information on the other sets appearing in Theorem~\ref{thm:max-wm-di-wm}.

\medskip
A prominent role in higher Teichm\"uller theory is played by Anosov representations, 
a notion introduced by F.~Labourie in his study of Hitchin representations \cite{Labourie_anosov},
then studied for Hermitian Lie groups in \cite{Burger_Iozzi_Labourie_Wienhard, Burger_Iozzi_Wienhard_anosov}
and in greater generality in \cite{Guichard_Wienhard_anosov}. 
The property of a representation to be Anosov is defined 
with respect to a parabolic subgroup $P<G$. 
In the context of weakly maximal representations the interest lies on representations $\rho: \pi_1(\Sigma) \to G$ which are Anosov with respect to the parabolic group $Q$
which is the stabilizer in $G$ of a point in the Shilov boundary $\cs$ 
of the bounded symmetric domain realization of $\mathcal{X}$ 
\cite{Burger_Iozzi_Labourie_Wienhard, Burger_Iozzi_Wienhard_anosov, Guichard_Wienhard_anosov}. 
We will call such an Anosov representation {\em Shilov-Anosov}. 


In the case in which $\partial \Sigma = \emptyset$,
if we denote by $\hom_{\cs\text{-}An}(\pi_1(\Sigma),G)$ the set of such representations
and by $\hom^*_{\cs\text{-}An}(\pi_1(\Sigma),G)$ the subset of Shilov-Anosov representations
with positive Toledo invariant, we have 
\bqn
\small{\xymatrix{
\hom_{max}(\pi_1(\Sigma),G)\ar@{^{(}->}[r]
&\hom^*_{\cs\text{-}An}(\pi_1(\Sigma),G)\ar@{^{(}->}[r]
&\hom_{\cs\text{-}An}(\pi_1(\Sigma),G).
}}
\eqn
Here $\hom_{max}(\pi_1(\Sigma),G)$ is a union of components
and $\hom_{\cs\text{-}An}(\pi_1(\Sigma),G)$ is an open subset of $\hom(\pi_1(\Sigma),G)$,
\cite{Burger_Iozzi_Labourie_Wienhard, Burger_Iozzi_Wienhard_anosov, Guichard_Wienhard_anosov}. 

On the other hand we prove in Corollary \ref{cor:closure} the following:

\begin{thm}\label{cssn-wm} If $G$ is of tube type and $\partial\Sigma=\emptyset$,
then 
\bqn
\xymatrix{
\overline{\hom_{\cs\text{-}An}(\pi_1(\Sigma),G)}\ar@{^{(}->}[r]
&\hom_{wm}(\pi_1(\Sigma),G)\,.
}
\eqn
\end{thm}
If $G=\Sp(2n,\RR)$ with $n\geq2$, we will see that $\hom_{\cs\text{-}An}(\pi_1(\Sigma),G)$ is open but not closed. The study of the closure of the set of Anosov representations is a completely open problem. Since the set of weakly maximal representations is closed (see Corollary~\ref{cor:wm_closed}), 
it provides a natural framework in which to study limits
of Shilov-Anosov representations. We refer the reader to \S~\ref{subsec:limit} for examples
of weakly maximal representations into $\Sp(2n,\RR)$ for $n\geq6$, that are not Shilov-Anosov, 
but are the limit of Shilov-Anosov representations.

For a closed surface $\Sigma$ and a Hermitian group $G$ of tube type
we  summarize the relation of various special subsets of the representation variety  and results from \cite{Labourie_anosov, Burger_Iozzi_Wienhard_tight}
in the following diagram. Here we denote by $\hom_{red}(\pi_1(\Sigma), G))$ 
the set of representations with reductive Zariski closure and by
$\hom_{Hitchin}(\pi_1(\Sigma), G))$ the Hitchin components in case $G$ is locally
isomorphic to 
$\Sp(2n,\RR)$ (and the empty set otherwise). Then we have the following
inclusions:
\bqn
\xymatrix{
                      &                 & \hom_{\cs\text{-}An} \ar@{^{(}->}[r]&
\hom_{wm}&\\
\hom_{Hitchin}  \ar@{^{(}->}[r] &\hom_{max} \ar@{^{(}->}[r] \ar@{^{(}->}[dr]&
\hom^*_{\cs\text{-}An} \ar@{^{(}->}[r] \ar@{^{(}->}[u]& \hom^*_{wm} \ar@{^{(}->}[r] \ar@{^{(}->}[u]&
\hom_{d,i}\\
&& \hom_{tight} \ar@{^{(}->}[r] & \hom_{red}
}
\eqn

\section{The Toledo Invariant}\label{sec:tolinv}
We recall here the definition of the Toledo invariant in a general context and 
indicate how the Milnor--Wood inequality follows from known isometric isomorphisms in bounded cohomology.

\medskip
Let $\Sigma$ be a compact oriented surface with (possibly empty) boundary $\partial\Sigma$.
Let $G$ be a locally compact group, $\k\in\hcb^2(G,\RR)$ a fixed continuous bounded class and
$\rho:\pi_1(\Sigma)\to G$ a homomorphism.  
Recall that there is an isometric isomorphism
\bqn
\xymatrix@1{
g_\Sigma:\hb^2(\pi_1(\Sigma),\RR)\ar[r]
&\hb^2(\Sigma,\RR)
}
\eqn
whose existence in general is due to Gromov \cite{Gromov_bounded};
in our case, the universal covering $\widetilde\Sigma$ is contractible and the existence
and isometric property of $g_\Sigma$ are easily established.  Next, the inclusion of pairs
$(\Sigma,\emptyset)\hookrightarrow(\Sigma,\partial\Sigma)$ gives rise to a map
\bqn
\xymatrix@1{
j_{\partial\Sigma}:\hb^2(\Sigma,\partial\Sigma,\RR)\ar[r]
&\hb^2(\Sigma,\RR)
}
\eqn
where the left hand side refers to bounded relative cohomology.
Since every connected component of $\partial\Sigma$ is a circle and 
hence has amenable fundamental group, 
the map $j_{\partial\Sigma}$ is an isometric isomorphism, \cite[Theorem~1.2]{BBFIPP}.
We hence define
\bqn
\T_\k(\rho):=\langle(j_{\partial\Sigma})^{-1}g_\Sigma(\rho^\ast(\k)),[\Sigma,\partial\Sigma]\rangle\,,
\eqn
where $[\Sigma,\partial\Sigma]$ is the relative fundamental class.  
Since $g_\Sigma$ and $j_{\partial\Sigma}$ are isometries, we deduce that
\bqn
|\T_\k(\rho)|\leq\|\rho^\ast(\k)\|\,\|[\Sigma,\partial\Sigma]\|_1\,,
\eqn
where the second factor refers to the norm in relative $\ell^1$-homology.
Since $\|[\Sigma,\partial\Sigma]\|_1=2|\chi(\Sigma)|$ and the pullback is norm decreasing, then 
\bq\label{eq:mw}
|\T_\k(\rho)|\leq 2\|\rho^\ast(\k)\|\,|\chi(\Sigma)|\leq 2\|\k\|\,|\chi(\Sigma)|\,.
\eq
In view of \eqref{eq:mw}, we give the following definition:

\begin{defi}\label{defi:k-wealy-maximal}A representation $\rho:\pi_1(\Sigma)\to G$ is 
$\k$-{\em weakly maximal} if $\T_\k(\rho)=2\|\rho^\ast(\k)\|\,|\chi(\Sigma)|$.
\end{defi}

\begin{rem} If $G$ is of Hermitian type and $\k=\kgb$ is the bounded K\"ahler class, 
then $\|\kgb\|=\rk(G)/2$ (see \cite[\S~2.1]{Burger_Iozzi_Wienhard_tol}) 
and one obtains for the corresponding Toledo invariant $\T(\rho)$
the familiar Milnor--Wood inequality
\bqn
|\T(\rho)|\leq\rk(G)\,|\chi(\Sigma)|\,.
\eqn
This will be used in the examples in \S~\ref{sec:examples}, but otherwise
never in the paper the explicit computation of the norm is used.
\end{rem}

\section{A Characterization of  Weakly Maximal Representations}\label{sec:milnor_wood}
In this section we develop the general framework for the study of weakly maximal representations 
and establish some of their basic properties.  In particular we show that 
weak maximality of a representation $\rho:\pi_1(\Sigma)\to G$
is reflected in a relation between its lift to appropriate central extensions
of $\pi_1(\Sigma)$ and $G$ and canonically defined quasimorphisms
on those central extensions (cf. Proposition~\ref{prop:equiv}).

\subsection{The Central Extension of $\pi_1(\Sigma)$}\label{subsec:pione}
Let $h$ be any complete hyperbolic structure on the interior $\Sigma^\circ$ of $\Sigma$ 
compatible with the fixed orientation
and let $\rho_h:\pi_1(\Sigma)\to\PU(1,1)$ be the corresponding holonomy representation.
Then 
\bq\label{eq:bfc}
\ksib:=\rho_h^\ast(\kpub)\in\hb^2(\pi_1(\Sigma),\RR)
\eq
 is independent of the choice of $h$
and is called the {\em (real) bounded fundamental class} of $\Sigma$, 
\cite[\S~8.2]{Burger_Iozzi_Wienhard_tol}.
If $\widehat\Gamma$ denotes the central extension 
\bq\label{eq:central extension}
\xymatrix{
0\ar[r]
&\ZZ\ar[r]
&\widehat\Gamma\ar[r]^-{p_\Sigma}
&\pi_1(\Sigma)\ar[r]
&0
}
\eq
corresponding to the positive generator of $\h^2(\pi_1(\Sigma),\ZZ)$ 
if  $\partial\Sigma=\emptyset$ and $\widehat\Gamma=\pi_1(\Sigma)$ otherwise,
then $\rho_h$ lifts to a homomorphism $\widetilde\rho_h:\widehat\Gamma\to\widetilde{\PU(1,1)}$,
where we think of the universal covering $\widetilde{\PU(1,1)}$ as contained in the group 
$\homeoz{\RR}$ of increasing homeomorphisms of the real line $\RR$ commuting with integer translations.
Denote by $\tau:\widetilde{\PU(1,1)}\to\RR$ the Poincar\'e translation quasimorphism
defined by $\tau(f):=\lim_{n\to\infty}\frac{f^n(x)-x}{n}$ for $x\in\RR$.
Then the composition $\tau\circ\widetilde\rho_h$ is a homogeneous quasimorphism which is $\ZZ$-valued
\bq\label{eq:Zvalued}
\tau\circ\widetilde\rho_h:\widehat\Gamma\to\ZZ\,,
\eq
since every element in the image of $\rho_h$ has a fixed point.
Moreover, the homogeneous quasimorphism $\tau$ corresponds in real bounded
cohomology to the pullback, via the projection $p:\widetilde{\PU(1,1)}\to\PU(1,1)$, 
of the real bounded cohomology class $\kpub$
\bq\label{eq:d tau}
[d\tau]=p^\ast(\kpub)\,,
\eq
and hence, again in real bounded cohomology,
\bq\label{eq:hom_qm}
[d(\tau\circ\widetilde\rho_h)]=p_\Sigma^\ast(\ksib)\,.
\eq
While $\tau\circ\widetilde\rho_h$ depends on the hyperbolization $\rho_h$ and
on its lift $\widetilde\rho_h$, its restriction $(\tau\circ\widetilde\rho_h)|_{\Lambda}$
to $\Lambda:=[\widehat\Gamma,\widehat\Gamma]$ is independent of the above choices,
(as the equivalence in the Proposition~\ref{prop:equiv} below shows).  

\subsection{The Central Extension for a Locally Compact Group $G$}\label{subsec:G}
We now turn to the construction of the central extension (depending on a ``rational" class)
and of the associated quasimorphism for a general locally compact group.
The relation with \S~\ref{subsec:pione}  is outlined in \S~\ref{subsec:herm}.

\medskip
\begin{defi}\label{defi:k-rational}  We say that a bounded cohomology class $\k$ is {\em rational}
if there is an integer $n\geq1$ such that $n\k$ is representable 
by a bounded $\ZZ$-valued Borel cocycle.
\end{defi}

Let $\k$ be rational and let $c:G^2\to\ZZ$ be a normalized bounded Borel cocycle representing $n\k$.
Endow $G\times\ZZ$ with the group structure defined by the Borel map
\bqn
(g_1,n_1)(g_2,n_2):=(g_1g_2,n_1+n_2+c(g_1,g_2))
\eqn
and let $G_{n\k}$ denote the Borel group $G\times\ZZ$ endowed with 
the unique compatible locally compact group topology \cite{Mackey_57}.  
This gives the topological central extension
\bq\label{eq:central extension G}
\xymatrix@1{
 0\ar[r]
&\ZZ\ar[r]^i
&G_{n\k}\ar[r]^{p_{n\k}}
&G\ar[r]
&e\,.}
\eq
Then $f'_{n\k}(g,m)=\frac{1}{n}m$ is a Borel quasimorphism such that $df'_{n\k}$ represents 
$p_{n\k}^\ast(\k)$.  Its homogenization $f_{n\k}:G_{n\k}\to\RR$ is a continuous
homogeneous quasimorphism \cite[Lemma~7.4]{Burger_Iozzi_Wienhard_tol} such that 
\bq\label{eq:dfk}
[df_{n\k}]=p_{n\k}^\ast(\k)\,,
\eq
and 
\bq\label{eq:values of fk}
f_{n\k}(i(m))=\frac{1}{n}m\,.
\eq
%
%
The homogeneous quasimorphisms $f_{n\k}$ are well defined in the following sense.
Let $\zeta=n\k$ be the smallest integer multiple of $\k$ representable 
by a $\ZZ$-valued bounded Borel cocycle and let $\ell\in\ZZ$, $\ell\neq0$.
Then the map $G\times\ZZ\to G\times\ZZ$ defined by $(g,m)\mapsto(g,\ell m)$
induces a continuous homomorphism that identifies $G_\zeta$ with 
a closed subgroup of finite index in $G_{\ell\zeta}$.  Via this identification,
the quasimorphism $f_\zeta$ is the restriction to $G_\zeta$ of $f_{\ell\zeta}$.

\medskip
The following result describes the property of being weakly maximal
in terms of quasimorphisms and will be essential in the sequel.

\begin{prop}\label{prop:equiv} Let $G$ be a locally compact group and
$\k\in\hb^2(G,\RR)$ a rational class with $n\k$ integral.  
Let $\rho:\pi_1(\Sigma)\to G$ be a homomorphism
and $\widetilde\rho:\widehat\Gamma\to G_{n\k}$ some lift,
where $\widehat\Gamma$ and $G_{n\k}$ are defined 
respectively in \eqref{eq:central extension} and \eqref{eq:central extension G}.
The following are equivalent:
\be
\item $\rho$ is $\kappa$-weakly maximal;
\item there exists $\lambda\geq0$ such that $\rho^\ast(\k)=\lambda\ksib$;
\item there exists $\lambda\geq0$ and $\psi\in\hom(\widehat\Gamma,\RR)$ such that 
\bqn
f_{n\k}\circ\widetilde\rho=\lambda(\tau\circ\widetilde\rho_h)+\psi\,;
\eqn
\item there exists $\lambda\geq0$ such that 
\bqn
(f_{n\k}\circ\widetilde\rho)|_{\Lambda}=\lambda (\tau\circ\widetilde\rho_h)|_{\Lambda}\,.
\eqn
\ee
\end{prop}
\begin{rem} The constant $\lambda\geq0$ that appears in Proposition~\ref{prop:equiv}
is nothing but $\lambda=\frac{\T_\k(\rho)}{|\chi(\Sigma)|}$, as it can be easily seen
from the definition of the Toledo invariant in \S~\ref{sec:tolinv}.
\end{rem}
\begin{proof}[Proof of Proposition~\ref{prop:equiv}] 
The equivalence of (1) and (2) is \cite[Corollary~4.15]{Burger_Iozzi_Wienhard_htt}.
In fact, the implication (2)$\Rightarrow$(1) is clear and we recall here the reverse implication
for the convenience of the reader.

Let $\rho_h$ be a complete finite area hyperbolization on $\Sigma^\circ$ so that 
$\Delta:=\rho_h(\pi_1(\Sigma))$ is a lattice in $\PU(1,1)$.  The space $\hb^2(\Delta,\RR)$ 
is isometrically isomorphic to the space of $\Delta$-invariant alternating $\linfty$-cocycles on
$(\partial\DD)^3$ and by \cite[Proposition~4.13]{Burger_Iozzi_Wienhard_htt}
we have a linear form $\tb:\hb^2(\Delta,\RR)\to\RR$ given by 
\bqn
\int_{\Delta\backslash\PU(1,1)}\alpha(gx,gy,gz)\,d\mu(g)=\frac{\tb(a)}{2}\,\orn(x,y,z)\,,
\eqn
where $\alpha:(\partial\DD)^3\to\RR$ corresponds to $a\in\hb^2(\Delta,\RR)$,
$\orn$ is the orientation cocycle and $\mu$ is the invariant
probability measure on $\Delta\backslash\PU(1,1)$. 
We obtained then in \cite[Theorem~4.9]{Burger_Iozzi_Wienhard_htt} an analytic formula for $\T_\k(\rho)$ 
that reads
\bqn
\T_\k(\rho)=\tb((\rho_h^\ast)^{-1}(\rho^\ast(\k)))|\chi(\Sigma)|\,.
\eqn
If now $\rho$ is weakly maximal, we have from the equality
\bqn
\T_\k(\rho)=2\|\rho^\ast(\k)\||\chi(\Sigma)|
\eqn
that 
\bqn
\tb((\rho_h^\ast)^{-1}(\rho^\ast(\k)))=2\|\rho^\ast(\k)\|
\eqn
and hence, if $\alpha$ corresponds to $\rho^\ast(\k)$, that 
\bqn
\int_{\Delta\backslash\PU(1,1)}\alpha(gx,gy,gz)\,d\mu(g)=\|\alpha\|_\infty\,\orn(x,y,z)\,.
\eqn
This implies that $\alpha$ is a positive multiple of the orientation cocycle and hence 
$\rho^\ast(\k)=\l\ksib$ for some $\l\geq0$.

Turning to the equivalence of (2) and (3), observe that, modulo modifying
the lift $\widetilde\rho$ by a homomorphism with values in $\ZZ\hookrightarrow G_{n\k}$, 
the diagram
\bq\label{eq:commutativediagramm}
\xymatrix{
\widehat\Gamma\ar[r]^{\widetilde\rho}\ar[d]_{p_\Sigma}
&G_{n\k}\ar[d]^{p_{n\k}}\\
\pi_1(\Sigma)\ar[r]^\rho
&G
}
\eq
commutes.  This, together with 
the relation \eqref{eq:hom_qm} and the fact that $p_\Sigma^\ast$ induces 
an isomorphism in real bounded cohomology, 
proves the equivalence of (2) and (3).

Finally, by taking coboundaries on both sides of the equality in (4), 
one obtains that the bounded classes $[d(f_{n\k}\circ\widetilde\rho)]$
and $\lambda[d(\tau\circ\widetilde\rho_h)]$ coincide on $\Lambda$.
Since $\widehat\Gamma/\Lambda$ is amenable, 
we deduce the equality in $\widehat\Gamma$, which in turns implies (3).
\end{proof}

\subsection{The Hermitian Case}\label{subsec:herm}
Assume that $G$ is of Hermitian type and almost simple.
We apply the above discussion to the bounded K\"ahler class $\kgb$, which is rational,
\cite[\S~5]{Burger_Iozzi_supq}.
If $\k=n\kgb$ is any integer multiple represented by an integral cocycle,
the topological central extension
\bqn
\xymatrix@1{
 0\ar[r]
&\ZZ\ar[r]^i
&G_{\k}\ar[r]^{p_\k}
&G\ar[r]
&e}
\eqn
is not trivial.  
Since $\pi_1(G)$ is isomorphic to $\ZZ$ modulo torsion,
there is a unique connected central $\ZZ$-extension $\widehat{G}$
and, as a result, the connected component of the identity $(G_\k)^\circ$
is isomorphic to $\widehat{G}$.  We denote by $f_{\widehat{G}}:\widehat{G}\to\RR$
the continuous homogeneous quasimorphism corresponding to $f_\k$
under this isomorphism and hence obtain that 
\bq\label{eq:qmcheck}
[df_{\widehat{G}}]=p^\ast(\kgb)\,,
\eq
where $p=p|_{(G_\k)^\circ}:\widehat{G}\to G$ is the projection.
For example, if $G=\PU(1,1)$, the bounded K\"ahler class $\kpub$ is already integral
as it is the image of the bounded Euler class $\eb$ under the change of coefficients
$\hcb^2(\PU(1,1),\ZZ)\to\hcb^2(\PU(1,1),\RR)$.  Then $\widehat{G}$ is the universal 
covering $\widetilde{\PU(1,1)}$ and $f_{\widehat{G}}$ is the Poincar\'e translation
quasimorphism.

\section{A General Framework for Injectivity and Discreteness}\label{sec:discrete}
In this section we give two results on weakly maximal representations: one concerns
injectivity with general locally compact targets, and the other discreteness of the image
in the case of Hermitian targets.  In both cases we use the characterization of 
weakly maximal representations in \S~\ref{sec:milnor_wood}.  However,
while the characterization holds also if the Toledo invariant vanishes,
both results in this section are obtained under the assumption that $\T_\k(\rho)\neq0$. Discreteness requires in addition that $\T_\k(\rho)$ is rational. 
The rationality question will be further addressed in Section~\ref{sec:rationality}.

\medskip
\begin{prop}\label{prop:injective}
Let $\rho:\pi_1(\Sigma)\to G$ be $\k$-weakly maximal and assume that $\T_\k(\rho)\neq0$.
Then $\rho$ is injective.
\end{prop}
\begin{proof}
Applying Proposition~\ref{prop:equiv}, we have that 
\bqn
f_\k\circ\widetilde\rho=\lambda(\tau\circ\widetilde\rho_h)+\psi
\eqn
with $\lambda=\frac{\T_\k(\rho)}{|\chi(\Sigma)|}\neq0$.  Since $p_\Sigma^{-1}(\ker\rho)=\widetilde\rho^{-1}(i(\ZZ))$,
we deduce from $\lambda\neq0$ that $\tau\circ\widetilde\rho_h$ is a homomorphism on $p_\Sigma^{-1}(\ker\rho)$.
A formal argument using the central extension \eqref{eq:central extension}
restricted to $\ker\rho<\pi_1(\Sigma)$, the commutativity of the diagram
\bqn
\xymatrix{
\widehat\Gamma\ar[r]^{\widetilde\rho_h}\ar[d]_{p_\Sigma}
&\widetilde{\PU(1,1)}\ar[d]^p\\
\pi_1(\Sigma)\ar[r]^{\rho_h}
&\PU(1,1)\,,
}
\eqn
and \eqref{eq:d tau} implies that 
\bqn
(\rho_h|_{\ker\rho})^\ast(\kpub)=0\,.
\eqn
Since this implies that  $\rho_h({\ker\rho})$ is elementary (see next lemma),
we conclude that $\ker\rho$ is trivial.
\end{proof}

We provide a proof of the following easy lemma for ease of reference.

\begin{lemma}\label{lem:elementary}  Let $\pi:\Delta\to\PU(1,1)$ be a homomorphism 
such that $\pi^\ast(\kpub)=0$. Then $\pi(\Delta)$ is elementary, that is $\pi(\Delta)$
has a finite orbit in $\overline\DD$.
\end{lemma}

\begin{proof} Since $\pi^\ast(\kpub)=0$, we conclude from the exactness of the sequence
\bqn
\xymatrix{
\hom(\Delta,\RR/\ZZ)\ar[r]^-\delta
&\hb^2(\Delta,\ZZ)\ar[r]
&\hb^2(\Delta,\RR)}
\eqn
that $\pi^\ast(\eb)=\delta(\chi)$ for some homomorphism $\chi:\Delta\to\RR/\ZZ$.
In other words, if we consider $\chi$ as a homomorphism into the group of rotations
in $\PU(1,1)$, we have that $\chi^\ast(\eb)=\pi^\ast(\eb)$. 
In particular $(\pi|_{[\Delta,\Delta]})^\ast(\eb)=0$.
Hence, by Ghys' theorem \cite[Theorem~6.6]{Ghys_01},  
$[\Delta,\Delta]$ has a fixed point in $\partial\DD$.
There are then two cases, 
depending on whether or not $[\Delta,\Delta]$ has infinitely many fixed points.
In the first case $[\pi(\Delta),\pi(\Delta)]$ is trivial, 
that is $\pi(\Delta)$ is abelian.  
In the second case $\pi(\Delta)$ has a finite orbit in $\partial\DD$.
In either cases $\pi(\Delta)$ is elementary.
\end{proof}

\medskip

The following proposition gives a criterion to insure 
the discreteness of the image of a $\k$-weakly maximal representation.
In Theorem~\ref{thm:structure} we are going to show that 
such criterion is always satisfied if $\k$ is the bounded K\"ahler class.
\begin{prop}\label{prop:discrete}  
Let $G$ be a Lie group of Hermitian type and 
let $\k\in\hcb^2(G,\RR)$ be a rational class.
Assume that $\rho:\pi_1(\Sigma)\to G$ 
is $\k$-weakly maximal and that 
$\T_\k(\rho)\in\QQ^\times$.  
Then $\rho$ has discrete image.
\end{prop}

\begin{rem}\label{rem:Tol_inv_not_zero}  In the case in which $\partial\Sigma=\emptyset$,
we have automatically that $\T_\k(\rho)\in\QQ$, so that in this case the condition just reads
$\T_\k(\rho)\neq0$.
\end{rem}

\begin{proof} Let $\Gamma:=\pi_1(\Sigma)$.
It will be enough to show that $\rho([\Gamma,\Gamma])$ is discrete.
 In fact $\rho([\Gamma,\Gamma])$ is a normal subgroup of $\rho(\Gamma)$ and,
 if discrete, also of $\overline{\rho(\Gamma)}$;
 thus $\rho([\Gamma,\Gamma])$ centralizes the connected component
of the identity $\overline{\rho(\Gamma)}^\circ$ of $\overline{\rho(\Gamma)}$.
 If $\overline{\rho(\Gamma)}^\circ$ were not trivial, 
 we could find $\rho(\gamma)\in\rho(\Gamma)\cap\overline{\rho(\Gamma)}^\circ$
 with $\rho(\gamma)\neq e$: indeed $\overline{\rho(\Gamma)}^\circ$ is an open subgroup
 of $\overline{\rho(\Gamma)}$ and $\rho(\Gamma)$ is dense in $\overline{\rho(\Gamma)}$.
 But since, by Proposition~\ref{prop:injective}, $\rho$ is injective, 
 this would imply that $\gamma$ centralizes $[\Gamma,\Gamma]$. 
 This is a contradiction, which shows that $\rho(\Gamma)$ is discrete.

To show the discreteness of $\rho([\Gamma,\Gamma])$, 
we retain the notation of \S~\ref{subsec:herm}
and we define $L:=\overline{\rho([\Gamma,\Gamma])}$ and $L':=p^{-1}(L)$.
Observe that \eqref{eq:central extension G} and \eqref{eq:commutativediagramm}
imply that $p^{-1}(\rho[\Gamma,\Gamma])=\widetilde\rho([\widehat\Gamma,\widehat\Gamma])\ker\,p$.
Since $p$ is a covering, 
then $p^{-1}(\overline{\rho([\Gamma,\Gamma])})=\overline{p^{-1}(\rho([\Gamma,\Gamma]))}$,
so that
 \bq\label{eq:L'}
 L'=\overline{\widetilde\rho([\widehat\Gamma,\widehat\Gamma])\ker\,p}\,.
 \eq
 We apply now the implication (1)$\Rightarrow$(3) in Proposition~\ref{prop:equiv}
 to obtain that
 \bqn\label{eq:Tol_lambda}
 f_{n\k}\circ\widetilde\rho=\lambda(\tau\circ\widetilde\rho_h)+\psi\,,
 \eqn
 where $\lambda=\frac{\T_\k(\rho)}{|\chi(\Sigma)|}$ and $\psi\in\hom(\widehat\Gamma,\RR)$.
 Since $\tau\circ\widetilde\rho$ is $\ZZ$-valued,  we deduce that 
 \bqn
 f_{n\k}(\widetilde\rho([\widehat\Gamma,\widehat\Gamma]))
=\lambda\tau(\widetilde\rho_h([\widehat\Gamma,\widehat\Gamma]))
\subset\lambda\tau(\widetilde\rho_h(\widehat\Gamma))\subset\lambda\ZZ\,,
 \eqn
 which, together with \eqref{eq:values of fk} implies that 
 \bqn
 f_{n\k}(\widetilde{\rho}([\widehat\Gamma,\widehat\Gamma])\ker\, p)\subset\frac{1}{n}\ZZ+\lambda\ZZ\,.
 \eqn
 Since $\lambda\in\QQ$, then $\frac{1}{n}\ZZ+\lambda\ZZ$ is discrete in $\RR$.
Taking into account \eqref{eq:L'}, this implies that $f_{n\k}(L')\subset\RR$ is a discrete subset. 
Hence, if we denote by $(L')^\circ$ the identity component of $L'$, then
 \bq\label{eq:vanish}
 f_{n\k}|_{(L')^\circ}=0\,.
 \eq
This implies easily that $\k|_{L^\circ}$ vanishes as a bounded real class:
in fact, because of \eqref{eq:vanish} and \eqref{eq:dfk},
\bqn
0=df_{n\k}|_{(L')^\circ}=p^\ast(\k)|_{(L')^\circ}=(p|_{(L')^\circ})^\ast(\k|_{L^\circ})\,,
\eqn
where we used that,
since $p:G_{n\k}\to G$ is an open map, $p((L')^\circ)=L^\circ$.
 
 Consider now the subgroup $\Delta:=\rho^{-1}(\rho([\Gamma,\Gamma])\cap L^\circ)$ of $\Gamma$.
 Then $(\rho|_\Delta)^\ast(\k|_{L^\circ})=0$ and hence, 
 by hypothesis and definition of $\k$-weak maximality, $\lambda\ksib|_\Delta=0$.
 Since by hypothesis $\lambda\neq0$, then it must be that $\ksib|_\Delta=0$.
 Thus $\rho_h|_\Delta:\Delta\to\PU(1,1)$ satisfies
 the hypotheses of Lemma~\ref{lem:elementary} and hence is elementary.  
 Since $\Delta\vartriangleleft\pi_1(\Sigma)$, 
 this implies $\Delta$ is trivial and hence $\rho([\Gamma,\Gamma])\cap L^\circ$ is trivial
 as well.
 
 But since $L$ is a Lie group (as a Lie subgroup of $G$), 
 $L^\circ$ is open in $L$ and $\rho([\Gamma,\Gamma])\cap L^\circ$
 is dense in $L^\circ$.  Hence  $L^\circ$ is trivial, 
 and thus $L$, and consequently $\rho([\Gamma,\Gamma])$, is discrete.
 \end{proof}


\section{On the Radical Defined by a Bounded Class}\label{sec:radical}
In this section, given a locally compact group $L$ and a bounded rational class $\k\in\hcb^2(L,\RR)$,
we show the existence of a largest normal closed subgroup $\Rad_\k(L)$ on which the restriction
of the class vanishes.  We show moreover that the class $\k$ comes from a bounded real class
on the quotient $L/\Rad_\k(L)$, the radical of which is trivial.  If $L$ is a connected Lie group,
the quotient $L/\Rad_\k(L)$ is adjoint semisimple without compact factors.

\medskip
\begin{prop}\label{prop:closed} Let $L$ be a locally compact second countable group and 
let $\k\in\hcb^2(L,\RR)$ be a rational class.
There is a unique largest normal subgroup $N\vartriangleleft L$ 
with $\k|_N=0$ which, in addition, is closed.
\end{prop}
This relies on the following:

\begin{lemma}\label{lem:hom_qm}  Let $f:L\to\RR$ be a continuous homogeneous quasimorphism.
\be
\item There is a unique largest normal subgroup $N_1\vartriangleleft L$ with $f|_{N_1}=0$. 
\item There is a unique largest normal subgroup $N_2\vartriangleleft L$ on which $f$ is a homomorphism.
\ee
Both $N_1$ and $N_2$ are closed.
\end{lemma}

\begin{proof} Clearly (2) implies (1) with $N_1=\ker(f|_{N_2})$.

Let now $M_1,M_2$ be normal subgroups of $L$ such that $f|_{M_i}:M_i\to\RR$ is a homomorphism.
For $m_1\in M_1$ and $m_2\in M_2$, let $\chi(m_1m_2):=f(m_1)+f(m_2)$.
We claim that $\chi$ is well defined.  Indeed, if $m_1m_2=m'_1m'_2$ with $m_i,m'_i\in M_i$,
then $(m'_1)^{-1}m_1=m'_2m_2^{-1}\in M_1\cap M_2$.  
Thus $f((m'_1)^{-1}m_1)=f(m'_2m_2^{-1})$, which implies, 
taking into account  that $f$ is a homomorphism on $M_1$ and $M_2$, that 
\bqn
-f(m'_1)+f(m_1)=f(m'_2)-f(m_2)\,.
\eqn
This shows that $\chi$ is well defined.  Next we claim that $\chi$ is a homomorphism.
If $m_i,m'_i\in M_i$, and since $M_1$ is normal in $L$, we have
\bqn
\ba
  \chi((m_1m_2)(m'_1m'_2))
=&\chi(m_1(m_2m'_1m^{-1}_2)m_2m'_2)\\
=&f(m_1(m_2m'_1m^{-1}_2))+f(m_2m'_2)\\
=&f(m_1)+f(m_2m'_1m^{-1}_2)+f(m_2)+f(m'_2)\,.
\ea
\eqn
Since $f$ is a homogeneous quasimorphism, we have $f(m_2m'_1m^{-1}_2)=f(m'_1)$,
which implies that $\chi:M_1\, M_2\to\RR$ is a homomorphism.  Since $f$ is a quasimorphism,
we have in particular that for all $m_i\in M_i$
\bqn
|f(m_1m_2)-\chi(m_1m_2)|=|f(m_1m_2)-f(m_1)-f(m_2)|\leq C\,,
\eqn
for some constant $C$.  Thus the homogeneous quasimorphism $f|_{M_1M_2}$
is at finite distance from the homomorphism $\chi$ and hence $f|_{M_1M_2}=\chi$.

This shows the existence of a unique largest normal subgroup $N_2\vartriangleleft L$
on which $f$ is a homomorphism.  
By continuity of $f$,  the subgroup $N_2$ is closed.
\end{proof}

\begin{proof}[{Proof of Proposition~\ref{prop:closed}}]  
Let 
\bqn
\xymatrix@1{
 0\ar[r]
&\ZZ\ar[r]^i
&L_{n\k}\ar[r]^p
&L\ar[r]
&e}
\eqn
be the topological central extension in \eqref{eq:central extension G}
and  $f_{n\k}:L_{n\k}\to\RR$ the continuous
homogeneous quasimorphism such that 
$[df_{n\k}]=p^\ast(\k)\in\hcb^2(L_{n\k},\RR)$.
It is an easy verification that,
given a subgroup $N<L$, 
 the property $\k|_N=0$ is equivalent to the property that $f_{n\k}|_{p^{-1}(N)}$ is a homomorphism.
Together with Lemma~\ref{lem:hom_qm}, this concludes the proof of the proposition.
\end{proof}

\begin{defi} Let $\k\in\hcb^2(L,\RR)$ be a rational class.  
Denote by $\Rad_\k(L)$ the normal closed subgroup
of $L$ given by Proposition~\ref{prop:closed} and call it the $\k$-{\em radical} of $L$.
\end{defi}

\begin{cor}\label{cor:rad}  Let $L$ be a locally compact second countable group and let 
$\k\in\hcb^2(L,\RR)$ be a rational class.
\be
\item $\Rad_\k(L)\supset\Rad_a(L)$, where $\Rad_a(L)$ denotes the amenable radical of $L$.
\item Let 
\bq\label{eq:kahlerprojection}
\pi:L\to L/\Rad_\k(L)
\eq 
denote the canonical projection.  Then there is a unique
class $u\in\hcb^2(L/\Rad_\k(L),\RR)$ with $\pi^\ast(u)=\k$.  The restriction of $u$ to any
non-trivial closed normal subgroup of $L/\Rad_\k(L)$ does not vanish.
\ee
\end{cor}
\begin{proof} (1) follows from the fact that the bounded cohomology of an amenable group vanishes.

\medskip
\noindent
(2) The first assertion follows from the exactness of the short sequence
\bqn
\xymatrix{
0\ar[r]
&\hcb^2(L/N,\RR)\ar[r]
&\hcb^2(L,\RR)\ar[r]
&\hcb^2(N,\RR)\,,
}
\eqn
where $N\vartriangleleft L$ is any closed normal subgroup 
\cite[Theorem~4.1.1]{Burger_Monod_GAFA}.  The second assertion follows
from the first and the maximality of $\Rad_\k(L)$.
\end{proof}

We denote by $\hhcb^2(L,\ZZ)$ the cohomology of the complex of integer valued bounded Borel cochains on $L$.
If  $\vk\in\hhcb^2(L,\ZZ)$, we denote by $\vkr$ the image of $\vk$ in $\hcb^2(L,\RR)$.

\medskip
The above discussion applied to a general connected Lie group $L$
has the following nice consequences.

\begin{cor}\label{cor:adjoint}  Let $L$ be a connected Lie group and $\vk\in\hhcb^2(L,\ZZ)$.
Then $H:=L/\Rad_\vkr(L)$ is connected adjoint of Hermitian type 
and a direct product $H=H_1\times\dots\times H_n$ of simple non-compact factors.
\end{cor}
\begin{proof}
Since the quotient $L/\Rad_a(L)$ of $L$ by its amenable radical is adjoint semisimple
without compact factors, so is $H$ (Corollary~\ref{cor:rad}(1)).  
Let $H=H_1\times\dots\times H_n$ be the direct product decomposition into simple factors.
Let $u\in\hcb^2(H,\RR)$ be such that $\pi^\ast(u)=\vkr$, where $\pi:L\to H$
is the projection in \eqref{eq:kahlerprojection}.
According to Corollary~\ref{cor:rad}(2), $u|_{H_j}\neq0$ and hence $\hcb^2(H_j,\RR)\neq0$:
thus $H_j$ is of Hermitian type for every $1\leq j\leq n$ and hence so is $H$.
\end{proof}
The following definition will be needed in the next section.
\begin{defi}\label{defi:k-levi}
Let $L$ be a connected Lie group which admits a closed Levi factor
and $\vk\in\hhcb^2(L,\ZZ)$. 
A {\em $\vk$-Levi factor} of $L$ is a connected semisimple subgroup $S$ 
with finite center such that $\pi(S)=L/\Rad_\k(L)$, where $\pi$ is as in \eqref{eq:kahlerprojection},
and $\ker(\pi|_S)$ is the center $Z(S)$ of $S$.
\end{defi}
If $L_0$ is a Levi factor of $L$, a $\vk$-Levi factor
of $L$ is nothing but the product 
of the almost simple factors in $L$ whose image via $\pi$ is non-trivial.


\section{Rationality Questions}\label{sec:rationality}
The class $u$ in Corollary~\ref{cor:rad}(2) is {\em a priori} only a real class.
We show in this section that if $L$ is a closed subgroup of a group $G$
of Hermitian type and $\k$ is the restriction of a bounded rational class on $G$, 
then $u$ has nice integrality properties. 

If in particular $\k=\kgb|_L$ is the restriction of the bounded K\"ahler class,
then the class $u$ is the linear combination of the bounded K\"ahler classes
of the individual simple factors of $L/\Rad_{\kgb}(L)$ with rational coefficients
whose denominators are bounded by an integer depending only on $G$.

\medskip
If $G$ is a Lie group of Hermitian type, $L<G$ a closed subgroup and $\vk\in\hhcb^2(G,\ZZ)$, 
we denote by $\vkrl$ the restriction of $\vkr$ to $L$.
\begin{prop}\label{prop:adjoint} Let $G$ be a Lie group of Hermitian type and 
let $L<G$ be a closed connected subgroup that admits a closed Levi factor $S$.
Let $\vk\in\hhcb^2(G,\ZZ)$ and let $\pi:L\to H:=L/\Rad_\vkrl(L)$ denote the canonical projection.
Let $p_j:H\to H_j$ be the projection onto the simple factors of $H$ (see Corollary~\ref{cor:adjoint})
and let $\vkj$ be a generator of $\hhcb^2(H_j,\ZZ)$, for $j=1,\dots,n$.
If $u\in\hcb^2(H,\RR)$ is such that $\pi^\ast(u)=\vkrl$
(see Corollary~\ref{cor:rad}(2)), then
\bq\label{eq:ucoordinates}
u=\sum_{j=1}^n\lambda_jp_j^\ast(\vkrj)\
\eq
with
\bq\label{eq:lambdai}
\lambda_j\in\frac{1}{|Z(S)|}\ZZ\,,
\eq
where $Z(S)$ denotes the center of the $\vk$-Levi factor of $L$.
\end{prop}

\begin{proof}
According to \cite[Proposition~7.7\,(3)]{Burger_Iozzi_Wienhard_tol},
the set $\{p_j^\ast(\vkj):\,1\leq j\leq n\}$ is a basis of $\hhcb^2(H,\ZZ)$ 
corresponding to the basis $\{p_j^\ast(\vkrj):\,1\leq j\leq n\}$ of $\hcb^2(H,\RR)$.
It follows that 
\bqn
u=\sum_{j=1}^n\lambda_jp_j^\ast(\vkrj)\,,
\eqn
with $\lambda_j\in\RR$.

\medskip
In order to show that the $\lambda_j$ are rational with an universal bound on the denominator,
we consider the following diagram
\bq\label{rational}
\xymatrix{
  \hcb^2(H,\RR)\ar[r]^{(\pi|_S)^\ast_\RR}\ar@/^2pc/[rr]^{\pi^\ast}
&\hcb^2(S,\RR)
&\hcb^2(L,\RR)\ar[l]_{\res}\\
  \hhcb^2(H,\ZZ)\ar[u]\ar[r]_{(\pi|_S)^\ast_\ZZ}
&\hhcb^2(S,\ZZ)\ar[u]
&\hhcb^2(L,\ZZ)\,,\ar[u]\ar[l]^{\res}
}
\eq
where $\res$ is the restriction map in cohomology and the vertical arrows are the change of coefficients.

If $(\pi|_S)^\ast_\ZZ$ were surjective, again \cite[Proposition~7.7\,(3)]{Burger_Iozzi_Wienhard_tol}, 
the commutativity of \eqref{rational} and the fact that 
$(\pi|_S)^\ast_\RR$ is an isomorphism would readily imply the integrality of the $\lambda_j$.
This is however not necessarily true and the following lemma identifies explicitly 
the nature of the map $(\pi|_S)^\ast_\ZZ$.
\begin{lemma}\label{lem:rational} Let $\omega:S\to H$ be a surjective homomorphism
between connected semisimple Lie groups with finite center.
Then the map
\bqn
\xymatrix@1{
\omega^\ast_\ZZ:\hhcb^2(H,\ZZ)\ar[r]
&\hhcb^2(S,\ZZ)}
\eqn
is injective and 
\bqn
\operatorname{Image}(\omega^\ast_\ZZ)\supset\,|\ker\,\omega|\,\hhcb^2(S,\ZZ)\,,
\eqn
where $|\ker\,\omega|$ denotes the cardinality of $\ker\,\omega$ 
and $\hhcb^2(S,\ZZ)$ is considered as a $\ZZ$-module.
\end{lemma}
We postpone the proof of the lemma and use its conclusion with $\omega=\pi|_S$.
If $Z(S)=\ker(\pi|_S)$ denotes the center of $S$, the same argument as above,
applied to $|Z(S)|\vkrs\in(\pi|_S)^\ast_\ZZ(\hhcb^2(H,\ZZ))$ shows that
$|Z(S)|\,u$ is, in fact, in $\hhcb^2(H,\ZZ)$
and hence its coordinates $|Z(S)|\,\lambda_j$ are integers, $j=1,\dots,n$. 
\end{proof}

\begin{proof}[Proof of Lemma~\ref{lem:rational}]
If $M$ is any connected semisimple Lie group with finite center
and maximal compact $K_M$, Wigner's theorem asserts that
the restriction map
\bqn
\xymatrix@1{
\hhc^2(M,\ZZ)\ar[r]
&\hhc^2(K_M,\ZZ)
}
\eqn
is an isomorphism \cite{Wigner}.
This, together with  the fact that the comparison map 
\bqn
\xymatrix{
\hhcb^2(M,\ZZ)\ar[r]
&\hhc^2(M,\ZZ)}
\eqn
is an isomorphism \cite[Proposition~7.7]{Burger_Iozzi_Wienhard_tol},
implies the isomorphism
\bqn
\xymatrix@1{
j:\hhcb^2(M,\ZZ)\ar[r]
&\hhc^2(K_M,\ZZ)\,.
}
\eqn
Using now the long exact sequence in cohomology  
associated to the short exact sequence
\bqn
\xymatrix{
0\ar[r]
&\ZZ\ar[r]
&\RR\ar[r]
&\RR/\ZZ\ar[r]
&0\,,}
\eqn
and the fact that the terms of positive degree and real coefficients vanish,
we obtain that the connecting homomorphism
\bqn
\xymatrix@1{
\homc(K_M,\RR/\ZZ)\ar[r]^-{\delta_{K_M}}
&\hhc^2(K_M,\ZZ)}
\eqn
is an isomorphism. 

Given $\omega:S\to H$ as in the statement of the lemma,
let $K_H<H$ be a maximal compact subgroup.  
Then $K_S:=\omega^{-1}(K_H)$ is a maximal compact in $S$.
One verifies then that the diagram
\bqn
\xymatrix{
  \hhcb^2(H,\ZZ)\ar[r]^{\omega^\ast_\ZZ}\ar[d]_{\delta_{K_H}^{-1}\circ j}
&\hhcb^2(S,\ZZ)\ar[d]^{\delta_{K_S}^{-1}\circ j}\\
  \homc(K_H,\RR/\ZZ)\ar[r]_{\omega^\ast}
&\homc(K_S,\RR/\ZZ)\,,
}
\eqn
where the vertical arrows are isomorphisms and 
the bottom one is the precomposition with $\omega|_{K_S}$,
is commutative.  Then the assertion of the lemma follows
from the fact that $\ker\,\omega\subset K_S$ and
that 
$\omega|_{K_S}:K_S\to K_H$ is surjective.
\end{proof}

Recall that the bounded K\"ahler class $\kgb\in\hcb^2(G,\RR)$ 
of a Lie group $G$ of Hermitian type is rational 
\cite[\S~5]{Burger_Iozzi_supq}.

\begin{defi} Let $G$ be a Lie group of Hermitian type
and $\kgb\in\hcb^2(G,\RR)$ its bounded K\"ahler class.
The {\em K\"ahler radical} $\Rad_{\kgb}(L)$ of a closed subgroup $L<G$ 
is the $\kgb|_L$-radical of $L$.

If $L$ admits a closed Levi factor, a {\em K\"ahler-Levi factor} of $L$
is a $\kgb|_L$-Levi factor.
\end{defi}

\begin{cor}\label{cor:4.7} Let $G$ be a Lie group of Hermitian type
and let $L<G$ be a closed connected subgroup that admits a closed Levi factor.
Then the group $H:=L/\Rad_{\kgb}(L)$ is connected adjoint of Hermitian type, 
and admits a direct product decomposition $H=H_1\times\dots\times H_n$  
of simple non-compact factors.

Moreover if $\pi:L\to H$ and $p_j:H\to H_j$ 
are the canonical projections, then there exists an integer $\ell_G\geq1$
depending only on $G$ such that 
\bqn
\kgb|_L=\pi^\ast\left(\sum_{j=1}^n\nu_jp_j^\ast(\k_{H_j}^{\mathrm b})\right)\,,
\eqn
where $\nu_j\in\frac{1}{\ell_G}\ZZ$.
\end{cor}

\begin{rem}\label{rem:constants} Let $H$ be a Lie group of Hermitian type.
We denote by $q_H$ the smallest integer such that there exists $\vk\in\hcb^2(H,\ZZ)$
with $q_H\kappa_H^{\rm b}=\vkr$.  
If in addition $H$ is simple and $\vk_H$ is a generator of $\hcb^2(H,\ZZ)$,
then there exists an integer $m_H$ such that $\vk=m_H\vk_H$.  It follows that 
\bq\label{eq:generatorhi}
\vk_H^\RR=\frac{q_H}{m_H}\kappa_H^{\rm b}\,.
\eq
\end{rem}

\begin{proof}[Proof of Corollary~\ref{cor:4.7}] 
We apply Proposition~\ref{prop:adjoint} to $\vk\in\hcb^2(G,\ZZ)$ 
such that $\vkr=q_G\kgb$.  Then \eqref{eq:ucoordinates} and \eqref{eq:lambdai},
together with \eqref{eq:generatorhi}, show that 
\bqn
\kgb|_L=\pi^\ast\left(\sum_{j=1}^n\nu_jp_j^\ast(\k_{H_j}^{\mathrm b})\right)\,,
\eqn
with 
\bqn
\nu_j\in\frac{q_{H_j}}{q_Gm_{H_j}|Z(S)|}\ZZ\,,
\eqn
where $|Z(S)|$ is the cardinality of the center of a K\"ahler-Levi factor $S$.

Since there are only finitely many possible conjugacy classes of connected semisimple 
subgroups of $G$, we obtain the result.
\end{proof}


\section{Structure of Weakly Maximal Representations}\label{sec:structure}
In \S~\ref{sec:discrete} the discreteness of a $\k$-weakly maximal representation was proven 
under the assumption that the Toledo invariant $\T_\k(\rho)$ is rational.
In this section we prove that if $\k$ is the bounded K\"ahler class, this is always the case and that
the representation into the quotient of its Zariski closure by the K\"ahler radical is also discrete and injective.  
The definition and properties of the K\"ahler radical will be essential to show that the Toledo 
invariant of the representation into the quotient is also non-zero.  

An interesting feature of the proof of the rationality of the Toledo invariant is that 
it depends upon showing first that the single factors of the quotient 
by the K\"ahler radical are of tube type.

\medskip
Let $G$ be a Lie group of Hermitian type and $\kgb\in\hcb^2(G,\RR)$ its bounded K\"ahler class.
For ease of notation, we denote the Toledo invariant $\T_{\kgb}$ by
\bqn
\T:\hom(\pi_1(\Sigma),G)\to\RR\,.
\eqn
\begin{defi} A homomorphism $\rho:\pi_1(\Sigma)\to G$ is weakly maximal if
\bqn
\T(\rho)=2\|\rho^\ast(\kgb)\|\,|\chi(\Sigma)|\,,
\eqn
that is if $\rho$ is $\kgb$-weakly maximal in the sense of Definition~\ref{defi:k-wealy-maximal}.
\end{defi}

\begin{thm}\label{thm:structure}  Let $G=\gG(\RR)^\circ$ be of Hermitian type, 
where $\gG$ is a connected semisimple algebraic group defined over $\RR$,
and let $\rho:\pi_1(\Sigma)\to G$ be a weakly maximal homomorphism.
Let $L$ be the connected component of the real points of the Zariski closure
of the image of $\rho$
and let $\Gamma:=\rho^{-1}(L)$.
Assume that $\T(\rho)\neq0$.  Then:
\be
\item the group $L/\Rad_{\kgb}(L)$ is Hermitian of tube type;
\item there is an integer $\ell_G\geq1$ depending only on $G$ (see Corollary~\ref{cor:4.7}), such that 
\bqn
\T(\rho)\in\frac{|\chi(\Sigma)|}{\ell_G}\ZZ\,;
\eqn
\item the composition
\bqn
\xymatrix@1{
\Gamma\ar[r]^{\rho|_\Gamma}
&L\ar[r]^-\pi
&L/\Rad_{\kgb}(L)
}
\eqn
is injective with discrete image.
\ee
\end{thm}

Using that in a surface group there are no finite normal subgroups, one obtains immediately:

\begin{cor}\label{cor:inj_discrete} Let $\rho:\pi_1(\Sigma)\to G$ be a weakly maximal homomorphism with 
$\T(\rho)\neq0$.  Then $\rho$ is injective with discrete image.
\end{cor}

An important role in the proof is played by the generalized Maslov cocycle.
Recall that if $H$ is a connected Lie group of Hermitian type and $\cs$ the Shilov boundary
of the associated bounded symmetric domain, the {\em generalized Maslov cocycle}
$\b_\cs:\cs^3\to\RR$ is a bounded alternating $H$-invariant cocycle
constructed by J.-L. Clerc in \cite{Clerc_maslov_tube}.   
We will use it in two ways:  on the one hand it represents the bounded K\"ahler class
in a particular resolution useful to implement pullbacks in bounded cohomology (see \eqref{eq:resol});
on the other, through the relation \eqref{eq:http}
with the Hermitian triple product
\bqn
\<\<\,\cdot\,,\,\cdot\,,\,\cdot\,\>\>:\cs^{(3)}\to\RR^\times\backslash\CC^\times\,,
\eqn
it gives a criterion to detect when $H$ is of tube type, 
\cite{Burger_Iozzi_supq, Burger_Iozzi_Wienhard_kahler}

We refer the reader to 
\cite{Burger_Iozzi_app, Burger_Iozzi_supq} and 
\cite[\S~4.2]{Burger_Iozzi_Wienhard_kahler} for details.

%
%

\begin{proof}[Proof of Theorem~\ref{thm:structure}]  
We use heavily in this proof techniques developed in 
\cite{Burger_Monod_GAFA, Burger_Iozzi_app, Burger_Iozzi_supq, Burger_Iozzi_Wienhard_kahler},
to which we refer the reader for details.

According to Corollary~\ref{cor:4.7}, the group
$H:=L/\Rad_{\kgb}(L)=H_1\times\dots\times H_n$ is of Hermitian type and 
$\kgb|_L=\pi^\ast(u)$, where 
\bqn
u=\sum_{j=1}^n\nu_jp_j^\ast(k_{H_j}^{\mathrm b})\,,
\eqn
and  $p_j:H\to H_j$ and $\pi:L\to H$ are the canonical projections.  
We set aside for the moment that the $\nu_j$ are rational, 
and will pick it up towards the end of the proof.

Let $\xi:=\pi\circ \rho|_\Gamma:\Gamma\to L\to H$
be the composition of $\pi$ with $\rho|_\Gamma$.  
Observe that $\Gamma$ is of finite index in $\pi_1(\Sigma)$.
It follows that, since $\rho$ is weakly maximal, that is 
$\rho^\ast(\kgb)=\lambda\ksib$ for $\lambda=\frac{\T(\rho)}{|\chi(\Sigma)|}$,
then
\bq\label{eq:xi_star}
\xi^\ast(u)=\lambda\k_{\Sigma'}^{\rm b}\,,
\eq
where $\Sigma'\to\Sigma$ is the finite covering corresponding to $\Gamma<\pi_1(\Sigma)$.

As usual, to realize the pullback $\xi^\ast$ in bounded cohomology, we use boundary maps.
According to \cite{Burger_Iozzi_app}, this is possible since
if $(\binfty(\cs^\bullet))$ denotes the complex of bounded alternating
Borel cocycles on $\cs$, then the class $[\beta_\cs]$ defined by the
generalized Maslov cocycle corresponds to the bounded K\"ahler class $\khb$ 
under the canonical map
\bq\label{eq:resol}
\xymatrix@1{\h^\bullet(\binfty(\cs^\bullet)^H)\ar[r]
&\hcb^\bullet(H,\RR)\,.
}
\eq

Likewise, if $\cs=\cs_1\times\dots\times\cs_n$ is the decomposition into a product, 
where $\cs_j$ is the Shilov boundary of $H_j$ and $\check p_j:\cs\to\cs_j$
is the projection, a standard cohomological argument
shows that the diagram
\bqn
\xymatrix{
\hcb^\bullet(H_j,\RR)\ar[r]^{p_j^\ast}
&\hcb^\bullet(H,\RR)\\
\h^\bullet(\binfty(\cs_j^\bullet)^{H_j})\ar[u]\ar[r]_{\check p_j^\ast}
&\hcb^\bullet(\binfty(\cs^\bullet)^H)\ar[u]
}
\eqn
commutes.  It follows that $u$ is represented, again via \eqref{eq:resol}, 
by the bounded Borel cocycle on $\cs^3$ defined by
\bqn
(x,y,z)\longmapsto\sum_{j=1}^n\nu_j\beta_{\cs_j}(x_j,y_j,z_j)\,.
\eqn

To recall the existence of the boundary map,
endow the interior of $\Sigma'$ with a complete hyperbolic metric of finite area so that $\Gamma$ is identified
with a lattice in $\PU(1,1)$.  Since $\xi(\Gamma)$ is Zariski dense in $H$, 
there is a $\Gamma$-equivariant measurable map $\varphi:\partial\DD\to\cs$.  
Because of the product structure of $\cs$, the map $\varphi$ has the form
\bqn
\varphi_1\times\dots\times\varphi_n:\partial\DD\longrightarrow\cs_1\times\dots\times\cs_n\,,
\eqn
where $\varphi_j:\partial\DD\to\cs_j$ is measurable and equivariant with respect to $p_j\circ\xi$.

Consequently, according to \cite{Burger_Iozzi_app} the pullback $\xi^\ast(u)$ is represented by
the following bounded measurable $\Gamma$-invariant cocycle
\bqn
(x,y,z)\longmapsto\sum_{j=1}^n\nu_j\beta_{\cs_j}(\varphi_j(x),\varphi_j(y),\varphi_j(z))\,,
\eqn
for almost all $(x,y,z)\in(\partial\DD)^3$.

Using now again that the bounded fundamental class of $\Sigma$ is represented by $\beta_{\partial\DD}$,
it follows from \eqref{eq:xi_star} that 
\bq\label{eq:sum}
\lambda\beta_{\partial\DD}(x,y,z)
=\sum_{j=1}^n\nu_j\beta_{\cs_j}(\varphi_j(x),\varphi_j(y),\varphi_j(z))
\eq
for almost all $(x,y,z)\in(\partial\DD)^3$.  

Recall now that the Hermitian triple product is an $H$-invariant real algebraic map 
on $(\cs)^3$ whose relation with the Maslov cocycle is given by
\bq\label{eq:http}
\<\<x,y,z\>\>_{\cs}\equiv e^{i\pi d_\cs\beta_\cs(x,y,z)}\mod\RR^\times\,,
\eq
where $d_{\cs}$ is an integer given in terms of the root system of $H$,
\cite{Burger_Iozzi_Wienhard_kahler}.  We will use this relation for 
the Hermitian triple product on each of the $H_j$.
To uniformize the exponents, 
since $\nu_j\in\QQ^\times$, we can pick an integer $m\geq1$ such that 
$m\nu_j=n_j d_j$ for some $n_j\in\ZZ$, $1\leq j\leq n$.
Multiplying \eqref{eq:sum} by $m$, exponentiating and using that 
$\beta_{\partial\DD}$ takes only $\pm1/2$ as values,
we obtain
\bq\label{eq:finite values}
e^{i\pi m\lambda(\pm1/2)}=\prod_{j=1}^n\langle\langle{}\varphi_j(x),\varphi_j(y),\varphi_j(z)\rangle\rangle_{\cs_j}^{n_j}\,,
\eq
for almost every $(x,y,z)\in(\partial\DD)^3$.
If $x_j,y_j,z_j\in\cs$, we define now 
\bqn
R((x_j),(y_j),(z_j)):=\prod_{j=1}^n\langle\langle{}x_j,y_j,z_j\rangle\rangle_{\cs_j}^{n_j}\,.
\eqn
Because of \eqref{eq:finite values} the function $R$ takes at most two values on $\varphi(\partial\DD)^3$;
since the latter set is Zariski dense in $\cs^3$ and $R$ is rational,
it takes on at most two values on $\cs^3$.  
The factors $\langle\langle{}x_j,y_j,z_j\rangle\rangle_{\cs_j}^{n_j}$ 
are pairwise independent rational functions,
therefore each of these factors itself can take only finitely many values. 
As previously recalled, it follows from \cite{Burger_Iozzi_Wienhard_kahler} 
that the corresponding groups are of tube type,
thus showing assertion (1).

Furthermore, since $\beta_{\cs_j}$ takes only values in $\frac12\ZZ$
\cite[Theorem~4.3]{Clerc_Orsted_TG},
it follows from \eqref{eq:sum} and Corollary~\ref{cor:4.7} 
that $\lambda\in\frac{1}{\ell_G}\ZZ$.
This shows assertion (2).  

The third assertion follows from 
Proposition~\ref{prop:injective}\,(2) and Proposition~\ref{prop:discrete}.
\end{proof}


\section{Weakly Maximal Representations and 
Relations with Other Representation Varieties}\label{sec:rep_var}
In this section we first show that the set of weakly maximal representations and 
the set of weakly maximal representations with non-zero Toledo number are closed. 
Then we examine the relationship of weakly maximal representations 
with Shilov-Anosov representations. 

\medskip
Returning to our general framework in \S~\ref{sec:milnor_wood}, let $G$ be 
locally compact and second countable, $\Gamma$ a discrete group and 
$\k\in\hcb^2(G,\RR)$ a fixed class.  We define now a (Hausdorff) topology
on $\hb^2(\Gamma,\RR)$ with respect to which the map
\bq\label{eq:continuous}
\ba
\hom(\Gamma,G)&\to\hb^2(\Gamma,\RR)\\
\rho\qquad&\longmapsto\,\rho^\ast(\k)
\ea
\eq
will be continuous. To this purpose recall that $\hb^2(\Gamma,\RR):=\ker\,\delta^2/\im\,\delta^1$,
where
\bqn
\xymatrix@1{
0\ar[r]
&\RR\ar[r]
&\ell^\infty(\Gamma)\ar[r]^-{\delta^1}
&\ell^\infty(\Gamma^2)\ar[r]^-{\delta^2}
&\ell^\infty(\Gamma^3)\ar[r]^-{\delta^3}
&\dots}
\eqn
denotes the inhomogeneous bar resolution.
If we endow each $\ell^\infty(\Gamma^n)$ with the weak-$\ast$ topology as the dual of $\ell^1(\Gamma^n)$,
then:

\begin{lemma}  $\hb^2(\Gamma,\RR)$ is a Hausdorff topological vector space
with the quotient weak-$\ast$ topology.
\end{lemma}
\begin{proof} It is clear that $\ker\,\delta^2$ is weak-$\ast$ closed.
Since $\hb^2(\Gamma, \RR)$ is a Banach space, $\im\,\delta^1$ is closed
and hence, by \cite[Theorem~4.14]{Rudin_FA}, $\im\,\delta^1$ 
is weak-$\ast$ closed as well.
\end{proof}

\begin{lemma} The function in \eqref{eq:continuous}
is continuous with respect to the weak-$\ast$ topology on $\hb^2(\Gamma,\RR)$.
\end{lemma}
\begin{proof}  Let $c_\k:G^2\to\RR$ be a bounded continuous inhomogeneous cocycle representing $\k$.
Let $(\rho_n)_{n\geq1}$ be a sequence of elements in $\hom(\Gamma,G)$ converging to $\rho$.
Setting
\bqn
\ba
c_n(x,y):=&c_\k(\rho_n(x),\rho_n(y))\\
c(x,y):=&c_\k(\rho(x),\rho(y))\,,
\ea
\eqn
we have that $c_n\in\ell^\infty(\Gamma^2)$ (respectively $c\in\ell^\infty(\Gamma^2)$) represent
$\rho_n^\ast(\k)$ (respectively $\rho^\ast(\k)$).
In addition
\be
\item $c_n\to c$ pointwise on $\Gamma^2$, and
\item $\|c_n\|_\infty$ and $\|c\|_\infty$ are bounded by $\|c_\k\|_\infty$.
\ee
Then if $f\in\ell^1(\Gamma^2)$, we have that
\be
\item $fc_n\to fc$ pointwise, and
\item $|f(x,y)c_n(x,y)|\leq|f(x,y)|\,\|\k\|_\infty$\,,
\ee
then the Dominated Convergence Theorem implies that 
\bqn
\int_{\Gamma^2}fc_n\to\int_{\Gamma^2}fc
\eqn
and shows therefore that $\rho_n^\ast(\k)\to\rho^\ast(\k)$ 
for the quotient topology in $\hb^2(\Gamma,\RR)$.
\end{proof}

If now $\Gamma=\pi_1(\Sigma)$, then

\begin{cor}\label{cor:wm_closed}
The set $\hom_{wm}(\pi_1(\Sigma),G)$ of weakly maximal representations 
is closed in $\hom(\pi_1(\Sigma),G)$.
\end{cor}
\begin{proof} It follows from Proposition~\ref{prop:equiv} that
\bqn
\ba
\hom_{wm}(\pi_1(\Sigma),G)=\{\rho\in\hom(\pi_1(\Sigma),G):\,\rho^\ast(\k)=t\ksib&\\
\text{ for some }t\in[0,\infty)\}&\,.
\ea
\eqn
Since $\hcb^2(\pi_1(\Sigma),\RR)$ with the weak-$\ast$ topology is Hausdorff, 
the subset $\{t\ksib:t\in[0,\infty)\}$ is closed.
The assertion follows then from the continuity of the map $\rho\mapsto\rho^\ast(\k)$.
\end{proof}

Let now set
\bqn
\ba
\hom_{wm}^\ast(\pi_1(\Sigma),G):=\{\rho:\pi_1(\Sigma)\to G:\,\rho\text{ is weakly maximal}&\\
\text{and }\T(\rho)\neq0\}\,.&
\ea
\eqn

\begin{cor}\label{cor:wm*} Let $G$ be a Lie group of Hermitian type.  
Then the set of weakly maximal representations with non-zero Toledo invariant
is closed in $\hom(\pi_1(\Sigma),G)$.
\end{cor}

\begin{proof} By Theorem~\ref{thm:structure}(2) we have 
\bqn
\hom_{wm}^\ast(\pi_1(\Sigma),G)
=\left\{\rho\in\hom_{wm}(\pi_1(\Sigma),G):\,\T(\rho)\geq\frac{|\chi(\Sigma)|}{l_G}\right\}\,,
\eqn
which implies the claim since $\rho\mapsto\T(\rho)$ is continuous, 
\cite[Proposition~3.10]{Burger_Iozzi_Wienhard_tol}, and $\hom_{wm}(\pi_1(\Sigma),G)$ is closed.
\end{proof}

Finally, we turn to the relation with Anosov representations.  
Let $\partial\Sigma=\emptyset$ and realize
$\pi_1(\Sigma)$ as cocompact lattice $\Gamma$ in $\PU(1,1)$ via a hyperbolization. 
If $G$ is, say, simple of tube type with Shilov boundary $\cs$, 
then a Shilov-Anosov representation $\rho:\Gamma\to G$ 
implies the existence of a (unique) continuous equivariant map
$\varphi:\partial\DD\to\cs$ with the additional property that for every $x,y\in\partial\DD$
with $x\neq y$, $\varphi(x)$ and $\varphi(y)$ are transverse.  
Let $(\partial\DD)^{3,+}$ be the connected set of distinct, positively oriented triples in $(\partial\DD)^3$
and $\cs^{(3)}$ the set of triples of pairwise transverse points in $\cs^3$.
Then
\bqn
\varphi\times\varphi\times\varphi:(\partial\DD)^{3,+}\to\cs^{(3)}
\eqn
must send $(\partial\DD)^{3,+}$ into a connected component of $\cs^{(3)}$. 
Thus if $\beta_{\cs}:\cs^3\to\RR$ denotes the generalized Maslov cocycle,
$\beta_{\cs}\circ\varphi^3$ is a multiple of the orientation cocycle.  
This and Proposition~\ref{prop:equiv} imply that the set
\bqn
\ba
\hom_{\cs\text{-}An}^+(\pi_1(\Sigma),G)^\circ:=
\{\rho:\pi_1(\Sigma)\to G:\,\rho\text{ is Shilov-Anosov}&\\
\text{ with }\T(\rho)\geq0\}&
\ea
\eqn
is contained in $\hom_{wm}(\pi_1(\Sigma),G)$.  Taking into account that the latter 
is closed, we get
\begin{cor}\label{cor:closure}
\bq\label{eq:closure}
\overline{\hom_{\cs\text{-}An}^+(\pi_1(\Sigma),G)}\subset\hom_{wm}(\pi_1(\Sigma),G)\,.
\eq
\end{cor}

\section{Examples}\label{sec:examples}

In this section we describe some examples of weakly maximal representations. 
\subsection{Maximal Representations}
Any maximal representation is a weakly maximal representation. 
Concrete examples of maximal representations are described in 
\cite{Burger_Iozzi_Wienhard_tol, Burger_Iozzi_Labourie_Wienhard, Guichard_Wienhard_invariants}.

\subsection{Embeddings of $\SL(2,\RR)$}
Special examples of weakly maximal representations arise from embeddings of $\SL(2,\RR)$. 
Consider a faithful representation $\rho_0: \pi_1(\Sigma) \to \SL(2,\RR)$ with discrete image, 
and let $\tau: \SL(2,\RR) \to G$ be an injective homomorphism into a Lie group of Hermitian type $G$.
Then the composition $\tau \circ\rho_0: \pi_1(\Sigma) \to G$ is a weakly maximal representation.

More generally if $H, G$ are Lie groups of Hermitian type, and $\tau:H \to G$ is a homomorphism 
such that $\tau^*(\kgb)$ is a multiple of $\khb$. 
Then the composition of a weakly maximal representation $\rho:\pi_1(\Sigma) \to H$ 
with $\tau$ is a weakly maximal representation into $G$.


\subsection{Shilov-Anosov Representations} 
We explained in \S~\ref{sec:rep_var} that any Shilov-Anosov representation into a Hermitian Lie group $G$  
of tube type with non-negative Toledo number is weakly maximal. Here we just give an example of such a representation into $\Sp(2n,\RR)$. 
Let again $\rho_i: \pi_1(\Sigma) \to \SL(2,\RR)$ $i= 1,\dots,n$ be faithful representations with discrete image. 
Consider the embedding $\tau: \SL(2,\RR)^n \to \Sp(2n,\RR)$ 
corresponding to a maximal polydisk in the bounded symmetric domain realization of $\Sp(2n,\RR)$. 
Then $\rho := \tau\circ (\rho_1,\cdots, \rho_n): \pi_1(\Sigma) \to \Sp(2n,\RR)$ is a Shilov-Anosov representation, 
and hence a weakly maximal representations. 
Since the set of Shilov-Anosov representations is open, 
any small deformation of such a representation is also weakly maximal. 
Similarly to the construction in \cite{Burger_Iozzi_Labourie_Wienhard} 
one can explicitly construct  bending deformations of the representation $\rho$
with Zariski dense image.

Note that $\rho$ if maximal if and only if the representations $\rho_i$ are orientation preserving for all $i= 1,\dots, n$.

\subsection{Cancelling Contributions}\label{sec:cancelling}
Let $G = G_1 \times \cdots  \times G_n$ be a semisimple Lie group of Hermitian type. 
If $\rho: \pi_1(\Sigma) \to G$ is a maximal representation,
then $\rho_i = p_i \circ \rho: \pi_1(\Sigma) \to G_i$, $ i = 1,\cdots, n$ are maximal representations, 
where $p_i$ denotes the projection onto the $i$-th factor. 

In order to illustrate that this does not hold true for weakly maximal representations 
consider an arbitrary representation $\rho_a: \pi_1(\Sigma) \to \SL(2,\RR)$ and 
denote by $\overline{\rho_a}: \pi_1(\Sigma) \to \SL(2,\RR)$ the composition of $\rho_a$ 
with the outer automorphism of $\SL(2,\RR)$ reversing the orientation. 
Let  $\rho_0: \pi_1(\Sigma) \to \SL(2,\RR)$ be a faithful representation with discrete image,
as before.
Then the representation 
$\rho = (\rho_0, \rho_a, \overline{\rho_a}) : \pi_1(\Sigma) \to \SL(2,\RR) \times \SL(2,\RR) \times \SL(2,\RR)$ is weakly maximal. 

\subsection{Limit of Shilov-Anosov Representations}\label{subsec:limit}  Let $\partial\Sigma=\emptyset$.
We have seen in Corollary~\ref{cor:closure} that any representation 
that is the limit of Shilov-Anosov representations is weakly maximal.
By choosing appropriately $\rho_a$ in \S~\ref{sec:cancelling}
it is easy to construct representations that are weakly maximal but not Shilov-Anosov.
While we do not know whether the containment in \eqref{eq:closure} is strict,
we give here an example to show that the containment 
\bqn
\hom_{\cs\text{-}An}^\ast(\pi_1(\Sigma),G)\subset\overline{\hom_{\cs\text{-}An}^\ast(\pi_1(\Sigma),G)}\,.
\eqn
is strict, that is an example of a representation that is the limit
of Shilov-Anosov representations but is not Shilov-Anosov.

To this purpose we consider the group $\Sp(2n,\RR)$ and
we denote for simplicity by $\Rr_d(n)$ the set
of representations $\rho:\pi_1(\Sigma)\to\Sp(2n,\RR)$ with Toledo invariant equal to $d$
and with $(\Rr_d(n))_{\cs\text{-}An}\subset\Rr_d(n)$
the subset of Shilov-Anosov representations.  

If $d=(g-1)n$, that is if $\Rr_{(g-1)n}(n)$ consists of maximal representations,
then by \cite{Burger_Iozzi_Labourie_Wienhard}  there is the equality
$(\Rr_{(g-1)n}(n))_{\cs\text{-}An}\equiv\Rr_{(g-1)n}(n)$.  On the other hand,
it is easy to construct a representation that has zero Toledo invariant
but is not Shilov-Anosov (by taking for example 
$\rho_a\oplus\overline{\rho_a}:\pi_1(\Sigma)\to\SL(2,\RR)\times\SL(2,\RR)\hookrightarrow\Sp(4,\RR)$,
where $\rho_a$ is any representation with dense image in $\SL(2,\RR)$).  
Since for $n\geq3$ the set $\Rr_0(n)$ is connected \cite[Theorem~1.1\,(1)]{GarciaPrada_Gothen_Mundet} 
and $(\Rr_0(n))_{\cs\text{-}An}$ is open, this shows that 
\bqn
(\Rr_0(n))_{\cs\text{-}An}\subsetneq\overline{(\Rr_0(n))_{\cs\text{-}An}}\,.
\eqn
We can use this fact as follows.
Recall that the direct sum of symplectic vector spaces leads to a canonical homomorphism
\bqn
\xymatrix{
\Sp(2n_1,\RR)\times\Sp(2n_2,\RR)\ar[r]
&\Sp(2(n_1+n_2),\RR)}
\eqn
and consequently to a continuous map
\bqn
\ba
\Rr_{d_1}(n_1)\times \Rr_{d_2}(n_2)&\to \Rr_{d_1+d_2}(n_1+n_2)\\
(\rho_1,\rho_2)&\mapsto\rho_1\oplus\rho_2\,.
\ea
\eqn
Observe moreover that $\rho_1\oplus\rho_2$ is Shilov-Anosov precisely if both $\rho_1$ and $\rho_2$ are.
In particular if $n_1\geq3$ we can consider a sequence $\rho_1^{(k)}\in(\Rr_0(n_1))_{\cs\text{-}An}$ whose
limit $\rho_1$ is not Shilov-Anosov.  Then for any $\rho_2\in(\Rr_d(n_2))_{\cs\text{-}An}$,
the sequence $\rho_1^{(k)}\oplus\rho_2$ converges to the representation $\rho_1\oplus\rho_2$ 
that is not Shilov-Anosov and has Toledo invariant equal to $d$.

\subsection{Weakly maximal representations with non-reductive Zariski closure}\label{subsec:nonred} 
Let $J_n$ be the $2n\times2n$ matrix with block entries $J_1=\begin{pmatrix}0&1\\-1&0\end{pmatrix}$.
We consider the subgroup of $\Sp(2(n+1),\RR)$ given by 
\bqn
Q=\left\{\begin{pmatrix} A&b&0\\0&1&0\\c&d&1\end{pmatrix}:\, A\in\Sp(2n,\RR)\text{ and }c={}^tbJ_nA\right\}\,,
\eqn
which is the semidirect product of $\Sp(2n,\RR)$ and the $n$-dimensional Heisenberg group $H_n$.

A map $\rho:\Gamma\to Q$ with entries $\pi(\gamma), b(\gamma), c(\gamma)$ and $d(\gamma)$ 
is a homomorphism if and only if $\pi:\Gamma\to\Sp(2n,\RR)$ is a homomorphism, 
$b:\Gamma\to\RR^{2n}$ is a $1$-cocycle (with the $\Gamma$-module structure on $\RR^{2n}$ is given by $\pi$) and 
\bq\label{eq:ex}
d(\gamma_1\gamma_2)={}^tb(\gamma_1)J_n\pi(\gamma_1)b(\gamma_2)+d(\gamma_1)+d(\gamma_2)\,.
\eq
To reinterpret \eqref{eq:ex}, we observe that there is a bilinear symmetric form on $\h^1(\Gamma,\pi)$
with values in $\h^2(\Gamma,\RR)$ obtained by composing the cup product 
$\h^1(\Gamma,\pi)\times\h^1(\Gamma,\pi)\to\h^2(\Gamma,\pi\otimes\pi)$ with the projection on trivial coefficients 
given by the invariant symplectic form;  
denoting by $Q_\pi$ the corresponding quadratic form, 
we have that 
\bqn
(\gamma_1,\gamma_2)\longmapsto{}^tb(\gamma_1)J_n\pi(\gamma_1)b(\gamma_2)
\eqn
is a representative of $Q_\pi([b])$ and the existence of a function $d$ satisfying \eqref{eq:ex}
amounts to $Q_\pi([b])=0$.

Let now $\Gamma=\pi_1(S)$, where $S$ is an oriented surface of genus $g$,
and let $\rho_i:\Gamma\to\Sp(2n_i,\RR)$,
where $\rho_1$ is in the Hicthin component of $\Sp(2n_1,\RR)$, 
while $\rho_2$ is the precomposition of a Hitchin representation into $\Sp(2n_2,\RR)$
with an orientation reversing automorphism of $\Gamma$.
Identifying $\h^2(\Gamma,\RR)$ with $\RR$ by means of the chosen orientation,
we have \cite{Goldman_email} that $Q_{\rho_1}$ is positive definite while $Q_{\rho_2}$ is negative definite.
Choose $[b_i]\in\h^1(\Gamma, \rho_i)$ both non-zero such that
\bq\label{eq:q}
Q_{\rho_1}([b_1])+Q_{\rho_2}([b_2])=0\,.
\eq
Set $\pi:=\rho_1\oplus\rho_2$, $b:=b_1+b_2$ and let $d$ be a solution of \eqref{eq:ex},
which exists by \eqref{eq:q}.  Then $\rho:\pi_1(S)\to\Sp(2(n+1),\RR)$ is a weakly maximal representation
with $\T(\rho)=n_1-n_2$, and the real Zariski closure of its image is
the semi direct product of $G_1\times G_2<\Sp(2n,\RR)$ with $H_n$; 
here $G_i$ is the real Zariski closure of the image of $\rho_i$.


\def\cprime{$'$} \def\cprime{$'$}
\providecommand{\bysame}{\leavevmode\hbox to3em{\hrulefill}\thinspace}
\providecommand{\MR}{\relax\ifhmode\unskip\space\fi MR }
\providecommand{\MRhref}[2]{%
  \href{http://www.ams.org/mathscinet-getitem?mr=#1}{#2}
}
\providecommand{\href}[2]{#2}

\vskip1cm

\end{document}